\documentclass[11pt]{article}
\usepackage{amssymb,latexsym,amsmath,amstext,amsthm,tabls,graphicx} 
\usepackage{comment, url}
\usepackage{fullpage}

\newcommand{\B}{\ensuremath{\mathcal{B}}}
\newtheorem{theorem}{Theorem}[section]
\newtheorem{lemma}[theorem]{Lemma}
\newtheorem{corollary}[theorem]{Corollary}
\newtheorem{example}{Example}[section]
\newtheorem{definition}{Definition}[section]
\newtheorem{remark}{Remark}[section]
\newcommand{\A}{\ensuremath{\mathcal{A}}}
\renewcommand{\B}{\ensuremath{\mathcal{B}}}
\newcommand{\Prob}{\ensuremath{\mathsf{Pr}}}

\begin{document}

\title{Combinatorial Solutions Providing Improved Security for the Generalized Russian Cards Problem}
\author{Colleen M.\ Swanson and 
Douglas R.\ Stinson\thanks{D.~Stinson's research is supported by NSERC discovery grant 203114-11}\\
David R.\ Cheriton School of Computer Science\\
University of Waterloo\\
Waterloo, Ontario N2L 3G1, Canada
}
\date{\today}

\maketitle

\begin{abstract} 
We present the first formal mathematical presentation of the generalized Russian cards problem, and provide rigorous security definitions that capture both basic and extended versions of weak and perfect security notions. In the generalized Russian cards problem, three players, Alice, Bob, and Cathy, are dealt a deck of $n$ cards, each given $a$, $b$, and $c$ cards, respectively. The goal is for Alice and Bob to learn each other's hands via public communication, without Cathy learning the fate of any particular card. The basic idea is that Alice announces a set of possible hands she might hold, and Bob, using knowledge of his own hand, should be able to learn Alice's cards from this announcement, but Cathy should not. Using a combinatorial approach, we are able to give a nice characterization of informative strategies (i.e., strategies allowing Bob to learn Alice's hand), having optimal communication complexity, namely the set of possible hands Alice announces must be equivalent to a large set of $t-(n, a, 1)$-designs, where $t=a-c$. We also provide some interesting necessary conditions for certain types of deals to be simultaneously informative and secure. That is, for deals satisfying $c = a-d$ for some $d \geq 2$, where $b \geq d-1$ and the strategy is assumed to satisfy a strong version of security (namely perfect $(d-1)$-security), we show that $a = d+1$ and hence $c=1$. We also give a precise characterization of informative and perfectly $(d-1)$-secure deals of the form $(d+1, b, 1)$ satisfying $b \geq d-1$ involving $d-(n, d+1, 1)$-designs.
\end{abstract}

\section{Introduction}
\label{sec: Introduction}

Suppose $X$ is a deck of $n$ cards, and we have three participants,
Alice, Bob and Cathy. Let $a+b+c = n$ and suppose that Alice is dealt a {\it hand} of $a$ cards, Bob is dealt a hand of $b$ cards and Cathy is dealt a hand of $c$ cards. These hands are random and dealt by some entity external to the scheme. We denote Alice's hand by $H_A$, Bob's hand by $H_B$ and Cathy's hand by $H_C$. Of course it must be the case that $H_A \cup H_B \cup H_C = X$. We refer to this as an {\it $(a,b,c)$-deal} of the cards.

For a positive   integer $t$, let $\binom{X}{t}$ denote the set of 
$\binom{n}{t}$ $t$-subsets of $X$. An {\it announcement} by Alice $\A$ is a subset of $\binom{X}{a}$. It is required that when Alice makes an announcement $\A$, the hand she holds is one of the $a$-subsets in $\A$. The goal of the scheme is that, after a deal has taken place and Alice has made an announcement, Bob should be able to determine Alice's hand, but Cathy should not be able to determine if Alice holds any particular card not held by Cathy. These notions will be formalized as we proceed. We remark that we focus on the scenario of Bob learning Alice's hand, although the original version of this problem is for Bob and Alice to learn each other's hand. We omit the latter case, since for any protocol whereby Bob may learn Alice's hand, Bob may then announce Cathy's hand publicly. This second step provides sufficient information for Alice to determine Bob's hand, without giving Cathy any more information than she previously had.

This problem was first introduced in the case $(a,b,c) = (3,3,1)$ in the 2000 Moscow Mathematics Olympiad. Since then, there have been numerous papers investigating the problem (called the Russian cards problem) and generalizations of it. Some are interested in card deal protocols that allow players to agree on a common secret without a given eavesdropper being able to determine this secret value. This area of research is especially interesting in terms of possible applications to key generation; see, for example~\cite{FPR91,FW91,FW93,FW93_2,FW96, MSN02, KMN04, ACDFJS11}. Others are concerned with analyzing variations of the problem using epistemic logic~\cite{D03,D05,DHMR06, CK08}. Duan and Yang~\cite{DY10} and He and Duan~\cite{HD11} consider a special generalization, with $n-1$ players each dealt $n$ cards, and one player (the intruder) dealt one card; the authors give an algorithm by which a dealer, acting as a trusted third party, can construct announcements for each player. Recently, there have been some papers that take a combinatorial approach~\cite{AAADH05,AD09,ACDFJS11,CDFJS}, which we discuss in some detail in Section~\ref{sec: Related Work}. 

We take a combinatorial point of view motivated by cryptographic considerations. To be specific, we provide definitions based on security conditions in the unconditionally secure framework, phrased in terms of
probability distributions regarding information available to the various players (analogous to Shannon's definitions relating to perfect secrecy of a cryptosystem). In particular, we provide a formal mathematical presentation of the generalized Russian cards problem. We introduce rigorous mathematical definitions of security, which in turn allows for systematic and thorough analysis of proposed protocols. We then give necessary conditions and provide constructions for schemes that satisfy the relevant definitions. Here there is a natural interplay with combinatorics.

\subsection{Overview of Contributions}
\label{subsec: Overview of Contributions}
The main contributions of our work are as follows:

\begin{itemize}
	\item We provide a formal mathematical presentation of the generalized Russian Cards problem. In particular, we define an \emph{announcement strategy} for Alice, which designates a probability distribution on a fixed set of possible announcements $\A_1, \A_2, \dots  , \A_m$ Alice can make. In keeping with standard practice in cryptography (i.e., Kerckhoff's principle), we assume that Alice's announcement strategy is public knowledge. Security through obscurity is not considered an effective security method, as secrets are difficult to keep; providing security under the assumption the adversary has full knowledge of the set-up of the given scheme is therefore the goal. This allows us to define the \emph{communication complexity} of the protocol to be $\log _2 m$ bits, since Alice need only broadcast the index $i$ of her chosen announcement, which is an integer between
$1$ and $m$. In order to minimize the communication complexity of the scheme, our goal will be to minimize $m$, the number of possible announcements. 
	\item We distinguish between \emph{deterministic} strategies, in which the hand $H_A$ held by Alice uniquely determines the index $i$ that she will broadcast, and \emph{non-deterministic}, possibly even biased announcement strategies. We are especially interested in strategies with uniform probability distributions, which we will refer to as \emph{equitable}.
	\item We examine necessary and sufficient conditions for a strategy to be \emph{informative for Bob}, i.e. strategies that allow Bob to determine Alice's hand. In particular, we give a lower bound on the communication complexity $m$ for informative strategies and provide a nice combinatorial characterization of strategies that meet this bound, which we term \emph{optimal} strategies.
	\item We provide the first formal security definitions that account for both \emph{weak} and \emph{perfect} security in an unconditionally secure framework. We remark that current literature focuses on weak security. In addition, we provide simpler, but equivalent combinatorial security conditions that apply when Alice's strategy is equitable. Here weak and perfect security are defined with respect to individual cards. If a scheme satisfies weak security (which we will term weak 1-security), Cathy should not be able to say whether a given card is held by Alice or Bob; if a scheme satisfies perfect security (which we will term perfect 1-security), each card is equally likely to be held by Alice.
	\item We use constructions and results from the field of combinatorial designs to explore strategies that are simultaneously informative and perfectly secure; this is especially useful for the case $c=1$. In particular, we analyze the case $c=a-2$ in detail, and show that strategies for $(a,b, a-2)$-deals that are simultaneously informative and perfectly secure must satisfy $c=1$. We also show a precise characterization between \emph{Steiner triple systems} and $(3,n-4,1)$-deals.
	\item We generalize our notions of weak and perfect security, which focus on the probability that individual cards are held by Alice, and consider instead the probability that a given set of cards (of cardinality less than or equal to $a$) is held by Alice. We consider deals satisfying $c=a-d$ and achieve parallel results to the $c=a-2$ case.
\end{itemize}

\subsection{Preliminary Notation and Examples}
\label{subsec: Preliminary Notation and Examples}

Alice will choose a set of announcements,
say $\A_1, \A_2, \dots  , \A_m$ such that every $H_A \in \binom{X}{a}$
is in at least one of the $m$ announcements. For $H_A \in \binom{X}{a}$, define $g(H_A) = \{ i : H_A \in \A_i\}$. Alice's {\it announcement strategy}, or more simply, {\it strategy},
consists of a probability distribution $p_{H_A}$ on $g(H_A)$, for every
$H_A \in \binom{X}{a}$. The set of announcements and probability distributions are 
fixed ahead of time and they are public knowledge. We will use the phrase {\it $(a,b,c)$-strategy} to denote a strategy for an $(a,b,c)$-deal. In addition, we will assume without loss of generality that $p_{H_A}(i) > 0$ for all $i \in g(H_A)$. To see this, note that if $p_{H_A}(i) = 0$ for some $H_A \in \binom{X}{a}$ and $i \in g(H_A)$, this means Alice will never choose $\A_i$ when she holds $H_A$. But since the set of announcements and probability distributions are public knowledge, Cathy also knows this, so there is no reason to have included $H_A$ in the announcement $\A_i$.

When Alice is dealt a hand $H_A \in \binom{X}{a}$, she randomly chooses an index
$i \in g(H_A)$ according to the 
probability distribution $p_{H_A}$. Alice broadcasts the integer $i$ to specify her announcement
$\A_i$.
Because the set of announcements and probability distributions are fixed and public, the
only information that is broadcast by Alice is the index $i$, which is an integer between
$1$ and $m$. Therefore we define the {\it communication complexity} of the protocol to be $\log _2 m$ bits.
In order to minimize the communication complexity of the scheme, our goal will be to minimize $m$, the number of possible announcements. 

If $|g(H_A)| = 1$ for every $H_A$, then we have a {\it deterministic} scheme, because the hand
$H_A$ held by Alice uniquely determines the index $i$ that she will broadcast.
That is to say, in a deterministic scheme, for any given hand, there is only one possible announcement
that is permitted by the given strategy. 

More generally, suppose there exists a constant $\gamma$ such that  $|g(H_A)| = \gamma $ for every $H_A$.
Further, suppose that every probability distribution $p_{H_A}$ is {\it uniform},
i.e., $p_{H_A}(i) = 1/\gamma$ for every $H_A$ and for every $i \in g(H_A)$.
We refer to such a strategy as a {\it $\gamma$-equitable strategy}. A deterministic scheme is just a  $1$-equitable strategy.

\begin{example}
\label{exam1}
Let $X = \{0,\ldots,6\}$. Figure \ref{fig1} presents a  partition of
$\binom{X}{3}$ that is  due to Charlie Colbourn and Alex Rosa (private communication).
This yields a 
deterministic $(3,3,1)$-strategy having  $m=6$ possible announcements.
\begin{figure}

\[
\begin{array}{|c|c|} \hline
 i & \A_i \\ \hline
  1 & \{0,1,3\}, \{1,2,4\}, \{2,3,5\}, \{3,4,6\}, \{0,4,5\}, \{1,5,6\}, \{0,2,6\}\\ \hline 
 2 & \{0,2,3\}, \{1,3,4\}, \{2,4,5\}, \{3,5,6\}, \{0,4,6\}, \{0,1,5\}, \{1,2,6\}\\ \hline 
 3 & \{0,2,4\}, \{0,3,5\}, \{1,2,3\}, \{0,1,6\}, \{1,4,5\}, \{2,5,6\}\\ \hline
 4 & \{0,1,2\}, \{2,3,4\}, \{4,5,6\}, \{1,3,5\}, \{0,3,6\} \\ \hline
 5 & \{1,2,5\}, \{0,5,6\}, \{1,4,6\}, \{0,3,4\}, \{2,3,6\} \\ \hline
 6 & \{3,4,5\}, \{0,1,4\}, \{0,2,5\}, \{2,4,6\}, \{1,3,6\} \\ \hline
 \end{array}
\]
\caption{A 
deterministic $(3,3,1)$-strategy having a set of six possible announcements}
\label{fig1}
\end{figure}
\end{example}

\begin{example}
\label{exam2}
Let $X = \{0,\ldots,6\}$. In Figure \ref{fig2}, we present a set of ten announcements found by Don Kreher (private communication). 
It can be verified that
every $3$-subset of $X$ occurs in exactly two of these announcements. Therefore we have
a $2$-equitable $(3,3,1)$-strategy.
\begin{figure}
\[
\begin{array}{|c|c|} \hline
 i & \A_i \\ \hline
 1 & \{2,5,6\}, \{2,3,4\}, \{1,4,5\}, \{1,3,6\}, \{0,4,6\}, \{0,3,5\}, \{0,1,2\}\\ \hline 
 2 & \{2,5,6\}, \{2,3,4\}, \{1,4,6\}, \{1,3,5\}, \{0,4,5\}, \{0,3,6\}, \{0,1,2\}\\ \hline 
 3 & \{3,4,5\}, \{2,4,6\}, \{1,3,6\}, \{1,2,5\}, \{0,5,6\}, \{0,2,3\}, \{0,1,4\}\\ \hline
 4 & \{3,4,5\}, \{2,4,6\}, \{1,5,6\}, \{1,2,3\}, \{0,3,6\}, \{0,2,5\}, \{0,1,4\}\\ \hline
 5 & \{3,4,6\}, \{2,3,5\}, \{1,4,5\}, \{1,2,6\}, \{0,5,6\}, \{0,2,4\}, \{0,1,3\}\\ \hline
 6 & \{3,4,6\}, \{2,3,5\}, \{1,5,6\}, \{1,2,4\}, \{0,4,5\}, \{0,2,6\}, \{0,1,3\}\\ \hline
 7 & \{3,5,6\}, \{2,4,5\}, \{1,3,4\}, \{1,2,6\}, \{0,4,6\}, \{0,2,3\}, \{0,1,5\}\\ \hline
 8 & \{3,5,6\}, \{2,4,5\}, \{1,4,6\}, \{1,2,3\}, \{0,3,4\}, \{0,2,6\}, \{0,1,5\}\\ \hline
 9 & \{4,5,6\}, \{2,3,6\}, \{1,3,4\}, \{1,2,5\}, \{0,3,5\}, \{0,2,4\}, \{0,1,6\}\\ \hline
 10 & \{4,5,6\}, \{2,3,6\}, \{1,3,5\}, \{1,2,4\}, \{0,3,4\}, \{0,2,5\}, \{0,1,6\}\\ \hline
\end{array}
\]
\caption{An equitable $(3,3,1)$-strategy having a set of ten possible announcements}
\label{fig2}
\end{figure}
\end{example}

\subsection{Organization of the Paper}
\label{subsec: Organization of the Paper}
In Section~\ref{sec: Informative and Secure Strategies}, we study and define the notions of informative, weakly 1-secure, and perfectly 1-secure strategies. In Section~\ref{sec: Simultaneously Informative and Secure Strategies}, we explore strategies that are simultaneously informative and either weakly or perfectly 1-secure, and include an analysis of perfectly 1-secure strategies with $c=a-2$ in Section~\ref{subsec: Strategies with $c=a-2$}. We present a generalization of the notions of weak and perfect 1-security and analyze the case of perfectly $(d-1)$-secure strategies satisfying $c=a-d$ in Section~\ref{sec: Generalized Notions of Security}. We conclude in Section~\ref{sec: Conclusion}.

\section{Informative and Secure Strategies}
\label{sec: Informative and Secure Strategies}

\subsection{Strategies that are Informative for Bob}
\label{subsec: Strategies that are Informative for Bob}

Let's first consider an $(a,b,c)$-deal from Bob's point of view, after hearing Alice's announcement.
Suppose that $H_B \in \binom{X}{b}$ and $i \in \{1, \dots , m\}$. Define
\[ \mathcal{P}(H_B,i) = \{ H_A \in \A_i : H_A \cap H_B = \emptyset \} .\]
$\mathcal{P}(H_B,i)$ denotes the set of {\it possible hands} that 
Alice might hold, given that Bob's hand is $H_B$ and Alice's announcement 
is $\A_i$.
Alice's strategy is {\it informative for Bob} provided that 
\begin{equation}
\label{inf-Bob}
| \mathcal{P}(H_B,i) | \leq 1 
\end{equation}
for all $H_B \in \binom{X}{b}$ and for all $i$.
In this situation, if Bob holds the cards in $H_B$ and Alice broadcasts  
$i$, then Bob can determine the set of $a$ cards that Alice holds.

If for a particular announcement $\A_i$ and any hand $H_B \in \binom{X}{b}$, we have $| \mathcal{P}(H_B,i) | \leq 1$, we say that $\A_i$ is an \emph{informative announcement}. This terminology is in keeping with previous work, which considers protocol characteristics only on the level of individual announcements.
 
The following result was shown by Albert et al.~\cite{AAADH05}, albeit using different terminology:
\begin{theorem}
\label{infor.thm}
The announcement $\A_i$ is informative for Bob if and only if there do not
exist two distinct sets $H_A, H_A' \in \A_i$ such that 
$| H_A \cap  H_A' | \geq a-c$.
\end{theorem}

\begin{proof}
Suppose there exist two distinct sets $H_A, H_A' \in \A_i$ such that 
$| H_A \cap  H_A' | \geq a-c$.
We have that $| H_A \cup  H_A' | \leq 2a - (a-c) = a+c = n - b$.
Hence, there exists $H_B \in \binom{X}{b}$ such that $H_B \cap (H_A \cup  H_A') = \emptyset$.
Then 
$ \{ H_A, H_A'\} \subseteq  \mathcal{P}(H_B,i) ,$ which contradicts (\ref{inf-Bob}).

Conversely, suppose $\{ H_A, H_A'\} \subseteq  \mathcal{P}(H_B,i)$, where $H_A \neq H_A'$.
Then $| H_A \cup  H_A' | \leq n-b = a+c$, and hence $| H_A \cap  H_A' | \geq a-c$.
\end{proof}

It follows from Theorem \ref{infor.thm} that the $(3,3,1)$-strategies 
presented in Examples \ref{exam1} and \ref{exam2}
are both informative for Bob, because 
$| H_A \cap  H_A' | \leq 1$ whenever $H_A$ and  $H_A'$ are two distinct sets in
the same announcement.  

We also have the following necessary condition.

\begin{corollary}
\label{NC-cor}
Suppose 
there exists a strategy for Alice that is informative
for Bob. Then $a > c$.
\end{corollary}

Furthermore, when $a>c$, we can derive a lower bound on the size of Alice's announcement.

\begin{theorem}
\label{bound.thm}
Suppose $a >c$ and 
there exists a strategy for Alice that is informative
for Bob. Then $m \geq \binom{n-a+c}{c}$.
\end{theorem}

\begin{proof}
Let $X' \subseteq X$ where $|X'| = a-c$. There are precisely $\binom{n-a+c}{c}$ 
$a$-subsets of $X$ that contain $X'$. These $a$-subsets must occur in different
announcements, by Theorem \ref{infor.thm}. Therefore $m \geq \binom{n-a+c}{c}$.
\end{proof}

In view of the above theorem,
an $(a,b,c)$-strategy for Alice that is informative for Bob is 
said to be {\it optimal} if $m = \binom{n-a+c}{c}$. We will next give a nice
combinatorial characterization of such optimal strategies. First, we require some
definitions from design theory.

\begin{definition}
Suppose that $t,v,k$ and $\lambda$ are positive integers with
$t \leq k < v$.
A {\em $t$-$(v,k,\lambda)$-design} is a pair $(X, \B)$, where $X$ is a set of
$v$ {\em points} and $\B$ is a multiset of $k$-element subsets of $X$ called {\em blocks},
such that every subset of $t$ points from $X$ occurs in precisely $\lambda$ blocks in 
$\B$. 
A $t$-$(v,k,\lambda)$-design, $(X, \B)$ is {\em simple} if every block in $\B$
occurs with multiplicity one.
\end{definition}

In a $2$-$(v,k,\lambda)$-design, every point occurs in 
exactly $\lambda (v-1) / (k-1)$ blocks and the total number of blocks is $\lambda v(v-1)/(k(k-1))$. 

A {\it Steiner triple system of order $v$} (denoted STS$(v)$) is
a $2$-$(v,3,1)$-design. It is well-known that an STS$(v)$ exists if and only if
$v \equiv 1,3 \bmod 6$, $v \geq 7$ (Theorem~4.18~of~\cite{CD},~p.~70). 

\begin{definition}
A {\em large set} of $t$-$(v,k,1)$-designs is a set of
$t$-$(v,k,1)$-designs, $(X,\B_1), \dots , (X,\B_N)$ (all of which have the
same point set, $X$), in which every $k$-subset of $X$ occurs as a block in precisely
one of the $\B_i$s. (Thus, the  $\B_i$s form a partition of $\binom{X}{k}$.)
It is easy to prove that there must be exactly $N = \binom{v-t}{k-t}$ designs in the large set.
\end{definition}

There are $v-2$ designs in a large set of STS$(v)$;
it is known that a  large set of STS$(v)$ exists if and only if
$v \equiv 1,3 \bmod 6$, $v \geq 9$ (Theorem~4.24~of~\cite{CD},~p.~70).

\begin{theorem}
\label{LS.thm}
Suppose that $a > c$.
An optimal $(a,b,c)$-strategy for Alice that is informative for Bob is equivalent to a 
large set of $t$-$(n,a,1)$-designs, where $t = a-c$. 
\end{theorem}

\begin{proof}
Suppose there exists a large set of $(a-c)$-$(n,a,1)$-designs. Then it is easy to see
that this immediately yields an optimal $(a,b,c)$-strategy for Alice that is informative for Bob.

Conversely, suppose there is an optimal $(a,b,c)$-strategy for Alice that is informative for Bob.
We need to show that every announcement is an $(a-c)$-$(n,a,1)$-design. Denote $t=a-c$ and
let $X'\subseteq X$, $|X'|=t$.  From the proof of Theorem \ref{bound.thm},
the $a$-subsets containing $X'$ occur in $\binom{n-a+c}{c}$ different announcements.
However, there are a total of $\binom{n-a+c}{c}$ announcements, so every announcement must
contain a block that contains $X'$.
\end{proof}

An optimal $(3,3,1)$-strategy would have $m = 5$. From Theorem \ref{LS.thm}, 
the existence of such a strategy would be equivalent to
a large set of five STS$(7)$. As mentioned above, it is known that this large set does not
exist. However, from Example \ref{exam1}, we obtain a $(3,3,1)$-strategy 
for Alice with $m=6$ that is informative for Bob. Thus we have proven the following.

\begin{theorem} 
The minimum $m$ such that there exists a $(3,3,1)$-strategy 
for Alice that is informative for Bob is $m=6$.
\end{theorem}

\subsection{Strategies that are Secure against Cathy}
\label{subsec: Strategies that are Secure against Cathy}

Now we consider security requirements for an $(a,b,c)$-strategy. Suppose that Alice makes an announcement 
$\A_i$ while trying to conceal information about her hand from Cathy. Necessarily  
Alice's hand is an $a$-subset in  $\A_i$.
In fact, Cathy knows that Alice's hand must be one of the 
$a$-subsets in the set $\mathcal{P}(H_C,i) = \{ H_A  \in \A_i : H_A \cap H_C = \emptyset \}$.
Therefore Cathy does obtain some partial 
information about Alice's hand.
However, it might be possible to prevent
Cathy from determining 
whether any individual card in $X \backslash H_C$ is held by Alice or by Bob. We define
two versions of this security property:

\begin{definition}
\mbox{\quad} 
\begin{enumerate}
\item Alice's strategy is {\em weakly $1$-secure against Cathy} provided that, 
for any announcement $i$, 
for any $H_C \in \binom{X}{c}$ such that $\mathcal{P}(H_C,i) \neq \emptyset$, 
and for any $x \in X \backslash H_C$, it holds that 
\[0 < \Prob [ x \in H_A | i,H_C] < 1.\]
Weak security means that, from Cathy's point of view, any individual card in $X \backslash H_C$
could be held by either Alice or Bob.

\item Alice's strategy is {\em perfectly $1$-secure against Cathy} provided that 
for any announcement $i$, 
for any $H_C \in \binom{X}{c}$ such that $\mathcal{P}(H_C,i) \neq \emptyset$, 
and for any $x \in X \backslash H_C$, it holds that 
\[ \Prob [ x \in H_A | i,H_C] = \frac{a}{a+b}.\]
Perfect security means that, from Cathy's point of view, the probability that 
any individual card in $X \backslash H_C$
is held by Alice is a constant. This probability must equal $a/(a+b)$ because Alice 
holds $a$ of the $a+b$ cards not held by Cathy.
\end{enumerate}
It is obvious that perfect $1$-security implies weak $1$-security.
\end{definition}

\noindent{\bf Remark:}
The condition $\mathcal{P}(H_C,i) \neq \emptyset$ is included to account for the possibility
that an announcement $i$ is not compatible with certain hands $H_C$ held
by Cathy. 

\medskip

The conditions for weak and perfect $1$-security depend on the probability distributions $p_{H_A}$
and the possible announcements. We will derive simpler, but equivalent, conditions 
of a combinatorial nature 
when Alice's strategy is equitable. First we state and prove a useful lemma which 
establishes that in an equitable strategy, from Cathy's point of view, any hand 
$H_A \in \mathcal{P}(H_C,i)$ is equally likely. 

\begin{lemma}
\label{equitable.lem}
Suppose that Alice's strategy is $\gamma$-equitable, Alice's announcement is $i$,
$H_C \in \binom{X}{c}$ and $H_A \in \mathcal{P}(H_C,i)$.  Then

\begin{equation}
\label{equitable.eq}
\Prob[H_A | H_C,i ] = \frac{1 }{ |\mathcal{P}(H_C,i)|}.
\end{equation}
\end{lemma}

\begin{proof}
We have
\[ \Prob[H_A | H_C,i ] =  \frac{\Prob[H_A , H_C,i ]}{\Prob[H_C,i ]}.\]
We can compute
\begin{eqnarray*}\Prob[H_A , H_C,i ] &=& \Prob[H_C | H_A , i ]\,\Prob[i| H_A  ]\,\Prob[H_A  ]\\
& = & \frac{1}{\binom{b+c}{c}} \times \frac{1}{\gamma} \times \frac{1}{\binom{n}{a}}.
\end{eqnarray*}
Similarly, we have 
\begin{eqnarray*} 
\Prob[H_C,i ] &=& \sum_{H'_A \in \mathcal{P}(H_C,i)}  
\Prob[H_C | H'_A , i ]\,\Prob[i | H'_A  ]\,\Prob[H'_A  ] \\
& = & |\mathcal{P}(H_C,i)| \times \frac{1}{\binom{b+c}{c}} \times \frac{1}{\gamma} \times \frac{1}{\binom{n}{a}}.
\end{eqnarray*}
The result follows.
\end{proof}

\begin{theorem}\label{equitable}
Suppose that Alice's strategy is $\gamma$-equitable. Then the following hold:
\begin{enumerate}
\item \label{weak} Alice's strategy is weakly $1$-secure against Cathy if and only if,
for any announcement $i$, 
for any $H_C \in \binom{X}{c}$ such that $\mathcal{P}(H_C,i) \neq \emptyset$, 
and for any $x \in X \backslash H_C$, it holds that
\[ 1 \leq | \{ H_A \in \mathcal{P}(H_C,i) : x \in H_A \} | \leq |\mathcal{P}(H_C,i))|  - 1.\]

\item \label{strong equitable}
Alice's strategy is perfectly $1$-secure against Cathy if and only if,
for any announcement $i$ and
for any $H_C \in \binom{X}{c}$ such that $\mathcal{P}(H_C,i) \neq \emptyset$, 
it holds that
\[ | \{ H_A \in \mathcal{P}(H_C,i) : x \in H_A \} = 
\frac{a \, |\mathcal{P}(H_C,i)| }{ a+b}\]
for any $x \in X \backslash H_C$.
\end{enumerate}
\end{theorem}

\begin{proof}
Since (\ref{equitable.eq}) holds, it immediately follows that 
\begin{equation}\label{single card probability}
\Prob [ x \in H_A | i,H_C] = \frac{|\{ H_A \in \mathcal{P}(H_C,i) : x \in H_A \}|}{ |\mathcal{P}(H_C,i)|}.
\end{equation}
Using Equation~\eqref{single card probability}, we observe that
\begin{equation*}
0 < \frac{|\{ H_A \in \mathcal{P}(H_C,i) : x \in H_A \}|}{ |\mathcal{P}(H_C,i)|}  < 1
\end{equation*}
holds if and only if 
\begin{equation*} 1 \leq |\{ H_A \in \mathcal{P}(H_C,i) : x \in H_A \}|
\leq |\mathcal{P}(H_C,i)| - 1.
\end{equation*}
This gives the first condition of the theorem.

Define $r_x = |\{ H_A \in \mathcal{P}(H_C,i) : x \in H_A \}|$.
Alice's strategy is perfectly $1$-secure against Cathy if and only if 
the value $\Prob [ x \in H_A | i,H_C]$ is independent of $x$. From (\ref{single card probability}), 
this  occurs if and only if
$r_x$ is independent of $x$.
We have that 
\[
\sum_{x \in X \backslash H_C} r_x = a \, |\mathcal{P}(H_C,i)|.
\]
There are $a+b$ terms  $r_x$ in the above sum. These terms are all equal if and only if they all have the
value $r = a \, |\mathcal{P}(H_C,i)| / (a+b)$.  This proves the second condition of the theorem.
\end{proof}

\noindent{\bf Remark:}
The above characterization of weak $1$-security for equitable strategies is equivalent to axioms {\bf CA2} and {\bf CA3} in \cite {AAADH05}. The characterization of perfect $1$-security for equitable strategies is equivalent to axiom {\bf CA4} in \cite {AD09}.

\medskip

It can be verified that the $(3,3,1)$-strategy in Example \ref{exam2} is perfectly $1$-secure 
against Cathy. However, the $(3,3,1)$-strategy in Example \ref{exam1} is only weakly
$1$-secure against Cathy.

Here is a sufficient condition for a strategy to be perfectly $1$-secure 
against Cathy.

\begin{lemma}
\label{2-design}
Suppose that each announcement in an equitable $(a,b,1)$-strategy is a $2$-$(n,a,\lambda)$-design. Then the strategy is perfectly $1$-secure against Cathy.
\end{lemma}

\begin{proof}
Given an announcement $\A_i$  and a point $x$, there are 
\[\lambda \left( \frac{n(n-1)}{a(a-1)} - \frac{n-1}{a-1} \right)\] 
blocks in 
$\A_i$ that do not contain $x$. Each of the points in 
$X \backslash \{x\}$ is contained in precisely 
\[ \lambda \left( \frac{n-1}{a-1} -1 \right)\] of these blocks.
\end{proof}

\section{Simultaneously Informative and Secure Strategies}
\label{sec: Simultaneously Informative and Secure Strategies}

In general, we want to find an $(a,b,c)$-strategy (for Alice) 
that is simultaneously informative for
Bob and (perfectly or weakly) $1$-secure against Cathy.

The following was first shown by Albert et al.~\cite{AAADH05} using a different proof technique:

\begin{theorem} 
If $a \leq c+1$, then there does not exist a strategy for Alice that is simultaneously informative for Bob and weakly $1$-secure against Cathy.
\end{theorem}

\begin{proof}
In view of Corollary \ref{NC-cor}, we only need to consider the
case $a = c+1$. In this case, any two $a$-subsets in an announcement must be disjoint, 
by Theorem \ref{infor.thm}.
For any announcement $\A_i$ and any $x\in X$, the definition of weak 1-security 
necessitates the existence of  a block
in $\A_i$ that contains $x$. It therefore follows that every 
$\A_i$ forms a partition of $X$ into $n/a$ blocks. 

Now, suppose that Alice's announcement is $\A_i$ and Cathy's hand is $H_C$.
There exists at least one $H_A \in \A_i$ such that $H_A \cap H_C \neq \emptyset$.
Now, $|H_C| <  |H_A|$, so there is a point $x \in H_A \backslash H_C$. The existence
of this point violates the requirement of weak 1-security.
\end{proof}

\begin{theorem}
Suppose $(a,b,c) = (3,n-4,1)$, where $n \equiv 1,3 \bmod 6$, $n > 7$. Then there exists an optimal strategy for Alice that is informative for Bob and perfectly $1$-secure against Cathy.
\end{theorem}

\begin{proof}
If $n \equiv 1,3 \bmod 6$, $n > 7$, then there exists a large set of disjoint
STS$(n)$ on an $n$-set $X$.
Theorem \ref{LS.thm} establishes that 
the resulting strategy is informative for Bob, because no announcement $\A_i$ 
(the set of blocks of an STS$(n)$) contains
two blocks that intersect in more than one point. 
Perfect 1-security follows immediately from Lemma \ref{2-design}.
\end{proof}

In the case $n=7$, there does not exist a large set of STS$(7)$, so we cannot construct an optimal $(3,3,1)$-strategy. However, Example \ref{exam2} provides us with an equitable strategy with $m=10$ and $\gamma = 2$ that is informative for Bob and perfectly $1$-secure against Cathy. This is because every announcement in this strategy is an STS$(7)$ and every $3$-subset occurs in exactly two announcements. Examples from the literature for this case typically only provide weak 1-security. Atkinson et al.~\cite{AD09} give a solution for the perfect 1-security case that requires a much larger communication complexity $m$ and also involves a complicated procedure in order to avoid card bias.

Next, we give a general method of obtaining equitable strategies from a single ``starting design''.
First we require some definitions.  

\begin{definition}
Suppose that $\mathcal{D} = (X,\B)$ is a $t$-$(v,k,\lambda)$-design.
An {\em automorphism} of $\mathcal{D}$ is a permutation $\pi$ of $X$ such that
$\pi$ fixes the multiset $\B$. The collection of all automorphisms of $\mathcal{D}$
is denoted  $\mathsf{Aut}(\mathcal{D})$;
it is easy to see that $\mathsf{Aut}(\mathcal{D})$
is a subgroup of the symmetric group $S_X$.
\end{definition}

\begin{theorem}
\label{aut.thm}
Suppose that $\mathcal{D} = (X,\B)$ is a $t$-$(n,a,1)$-design with $t = a-1$.
Then there exists a $\gamma$-equitable $(a,n-a-1,1)$-strategy with $m$ announcements that is informative for
Bob and perfectly $1$-secure against Cathy, where $\gamma = n! / | \mathsf{Aut}(\mathcal{D}) |$
and $m = \gamma (n-t)$.
\end{theorem}

\begin{proof}
Let the symmetric group $S_n$ act on $\mathcal{D}$. We obtain a set of designs isomorphic to $\mathcal{D}$. Every one of these designs is a $2$-design because $a \geq 3$, so the resulting scheme is perfectly $1$-secure against Cathy by Lemma~\ref{2-design}. Every design is also an $(a-1)$-design with $\lambda = 1$, so Theorem~\ref{infor.thm} implies the scheme is informative for Bob.

Finally, every block is in $n! / | \mathsf{Aut}(\mathcal{D})|$ of the resulting set of designs and the total number of designs is equal to $\gamma (n-a+1)$.
\end{proof}

\begin{example}\label{perfectly 2-secure}
It is known that there is a $3$-$(8,4,1)$-design having an automorphism group of order $1344$. (See, for example, result~13~of~Section~1.4 of Dembowski~\cite{D68}.) Theorem~\ref{aut.thm} thus yields a $30$-equitable $(4,3,1)$-strategy with $150$ announcements that is informative for Bob and perfectly $1$-secure against Cathy. However, in this particular case, we can do better. Don Kreher (private communication) has found a set of ten $3$-$(8,4,1)$-designs on a set of 
points $X = \{0, \dots , 7\}$ such that every $4$-subset of $X$ occurs in exactly two of these designs. Therefore we have a $2$-equitable $(4,3,1)$-strategy with ten announcements that is informative for
Bob and perfectly $1$-secure against Cathy. The set of $3$-$(8,4,1)$-designs can be constructed as follows:
Begin with  a $3$-$(8,4,1)$-design having the following set $\A_0$ of $14$ blocks: 

\[
\begin{array}{l}
\{3,4,5,6\}, \{2,5,6,7\}, \{2,3,4,7\}, \{1,4,5,7\}, \{1,3,6,7\}, \{1,2,4,6\}, \{1,2,3,5\}, \\
	\{0,4,6,7\}, \{0,3,5,7\}, \{0,2,4,5\}, \{0,2,3,6\}, \{0,1,5,6\}, \{0,1,3,4\}, \{0,1,2,7\}.
	\end{array}
	\]
	Define the permutation
	$\pi = (0,1)(2)(3,4,6,7,5)$ and let $\pi$ (and its powers) act on $\A_0$. 
\end{example}

\subsection{Strategies with $c = a-2$}
\label{subsec: Strategies with $c=a-2$}

In this section, we focus on $(a,b,a-2)$-deals that are simultaneously informative for Bob, equitable, and perfectly 1-secure against Cathy. Where possible, we weaken our assumption that the strategy is equitable and our assumption of perfect 1-security to achieve the given result. We do assume that the strategies discussed are informative throughout these results, although we may not re-emphasize this point in the intervening discussion. We begin with some notation.

Consider an $(a,b,c)$-deal and a corresponding announcement $\A_i$. For any point $x \in X$, we define the \emph{block neighborhood of $x$ with respect to $\A_i$}, denoted $B_x^i$, to be $B_x^i = \{H_A \in \A_i : x \in H_A\}$ and the \emph{neighborhood of $x$ with respect to $\A_i$}, denoted $N_i(x)$, to be 
\[N_i(x) = \left( \bigcup_{H_A \in B_x^i} H_A \right)\backslash \{x\}.\]
For ease of notation, if the choice of $\A_i$ is understood from context, we sometimes write $G$ for $G_i$, $N(x)$ for $N_i(x)$, and $B_x$ for $B_x^i$.

We can also extend these notions as follows. Define $B_{x_1,\ldots, x_t}^i = \{H_A \in \A_i : x_1,\ldots, x_t \in H_A\}$. That is, we say $B_{x_1,\ldots, x_t}^i$ is the \emph{block neighborhood of the set $\{x_1,\ldots, x_t\}$}. Similarly, define the \emph{neighborhood of $x_1,\ldots, x_t$}, denoted $N(x_1,\ldots, x_t)^i$, to be

\[N(x_1,\ldots, x_t)^i = \left(\bigcup_{H_A \in B_{x_1,\ldots, x_t}^i} H_A \right)\backslash \{{x_1,\ldots, x_t}\}.\]

If $i$ is understood from context, we refer to $B_{x_1,\ldots, x_t}$ and $N(x_1,\ldots, x_t)$.

We begin with a preliminary lemma concerning block neighborhoods. A simple consequence of an $(a,b,a-2)$-strategy being informative is that the intersection of any two distinct block neighborhoods has cardinality less than one.

\begin{lemma}\label{info.lemma}Consider an $(a,b,c)$-deal such that $a-c=2$ and a corresponding announcement $\A_i$. Suppose that Alice's strategy is informative for Bob. Then for any distinct $x,y \in X$, there is at most one hand $H_A \in \A_i$ such that $x,y \in H_A$. That is, $|B_x \cap B_y| \leq 1$.
\end{lemma}

\begin{proof}This follows directly from Theorem~\ref{infor.thm}.
\end{proof}

An interesting question concerns which hands are possible (i.e., occur with probability greater than zero) for Cathy for any given announcement $\A_i$. The following lemma shows that Cathy may hold a subset of any hand $H_A$ that appears in $\A_i$. We will then extend this result to show that Cathy's hand may consist of a subset of size $c-1$ from any hand $H_A$ appearing in $\A_i$, together with another card $z \notin H_A$.

\begin{lemma}\label{poss.hand1}Consider an $(a,b,c)$-deal such that $a-c=2$ and a corresponding announcement $\A_i$. Suppose that Alice's strategy is informative for Bob and weakly 1-secure against Cathy. Let $H_A \in \A_i$ and $Y \subset H_A$ satisfying $|Y| = c$. Then $\mathcal{P}(Y,i) \neq \emptyset$.
\end{lemma}

\begin{proof} 
We proceed by contradiction. Write $H_A = \{x_1,\ldots, x_a\}$ and $Y=\{x_1,\ldots,x_c\}$. Suppose $\mathcal{P}(Y,i) = \emptyset$. Then every hand of $\A_i$ intersects $Y$. In particular, by Lemma~\ref{info.lemma}, every hand in $\A_i$ (excluding $H_A$) contains exactly one element of $Y$.

Now, since Alice's strategy is weakly 1-secure against Cathy, there must be some $H_A' \in \A_i$ such that $H_A' \neq H_A$. By the above argument, $|H_A' \cap Y|=1$. Suppose, without loss of generality, $H_A'$ contains $x_1$. We will now use the existence of $H_A'$ to construct a possible hand for Cathy that would imply Alice holds $x_1$.

Define $Y' = \{x_2,\ldots,x_{a-1}\}$, so $|Y'| =c$. Then $\mathcal{P}(Y',i)$ consists of all the blocks in $\A_i$ containing $x_1$ (except for $H_A$). In particular, $H_A' \in \mathcal{P}(Y',i)$, so this set is nonempty, and therefore $Y'$ is a possible hand for Cathy. But if Cathy holds $Y'$, Alice must hold $x_1$, which contradicts the security assumption.
\end{proof}

Before we generalize Lemma~\ref{poss.hand1}, we will need the following result, which shows that, given a particular announcement $\A_i$ and choice of card $x$, at least two hands in $\A_i$ must contain $x$ and at least two hands in $\A_i$ must not contain $x$. We remark that Albert et al.~\cite{AAADH05} show a related result, namely that given an informative and weakly 1-secure announcement, each card must appear in at least $c + 1$ hands. 

\begin{lemma}\label{min4.lemma} Consider an $(a,b,c)$-deal such that $a-c=2$ and a corresponding announcement $\A_i$. Suppose that Alice's strategy is informative for Bob and weakly 1-secure against Cathy. Suppose $x \in X$. Then there are at least two hands of $\A_i$ containing $x$ and at least two hands of $\A_i$ that do not contain $x$.
\end{lemma}

\begin{proof}We proceed by contradiction. We first show there exist $H_A, H_A' \in \A_i$ which contain $x$. Now, there must be some $H_A \in \A_i$ satisfying $x \in H_A$, as otherwise $x$ is not held by Alice. Suppose all other hands of $\A_i$ do not contain $x$. By Lemma~\ref{poss.hand1}, we may pick $Y \subset H_A$, $|Y| = c$, where $\mathcal{P}(Y,i) \neq \emptyset$. That is, $Y$ is a possible hand for Cathy. But since $H_A \notin \mathcal{P}(Y,i)$, there must be some other $H_A' \in \A_i$ such that $x \in H_A'$. Otherwise, if Cathy holds $Y$, then Cathy knows $x$ is not held by Alice.

We show in a similar fashion there exist $H_A, H_A' \in \A_i$ which do not contain $x$. There must be some $H_A \in \A_i$ satisfying $x \notin H_A$, as otherwise $x$ must be held by Alice. Suppose all other hands of $\A_i$ contain $x$. By Lemma~\ref{poss.hand1}, we may pick $Y \subset H_A$, $|Y| = c$, where $\mathcal{P}(Y,i) \neq \emptyset$. That is, $Y$ is a possible hand for Cathy. But since $H_A \notin \mathcal{P}(Y,i)$, there must be some other $H_A' \in \A_i$ such that $x \notin H_A'$. Otherwise, if Cathy holds $Y$, then Cathy knows $x$ is held by Alice.
\end{proof}

\begin{lemma}\label{poss.hand2} Consider an $(a,b,c)$-deal such that $a-c=2$. Suppose that Alice's strategy is informative for Bob and weakly 1-secure against Cathy. Let $H_A \in \A_i$ and $Y \subset H_A$ satisfying $|Y| = c-1$. Let $z \in X$ such that $z \notin H_A$. Then $\mathcal{P}(Y \cup \{z\},i) \neq \emptyset$.
\end{lemma}

\begin{proof}
We proceed by contradiction. Write $H_A = \{x_1,\ldots, x_a\}$ and $Y=\{x_1,\ldots,x_{c-1}\}$. Suppose $\mathcal{P}(Y\cup \{z\},i) = \emptyset$. Then every hand of $\A_i$ intersects $Y \cup \{z\}$. In particular, by Lemma~\ref{info.lemma}, every hand in $\A_i$ (excluding $H_A$) contains at most one element of $Y$.

Now, by Lemma~\ref{min4.lemma}, there must be some $H_A' \in \A_i$ such that $H_A' \neq H_A$ and $z \notin H_A'$. By the above argument, $|H_A' \cap (Y \cup \{z\})|=1$. Suppose, without loss of generality, $H_A'$ contains $x_1$. We will now use the existence of $H_A'$ to construct a possible hand $Y'$ for Cathy that would imply Alice holds $x_1$.

Define $Y' = \{x_2,\ldots,x_{c},z\}= Y \cup \{z\} \cup \{x_c\}$, so $|Y'| =c$. Here we include $x_c \in Y'$ for technical reasons. We want $|Y'| = c$ and we need $H_A' \cap Y' = \emptyset$, which we have since Lemma~\ref{info.lemma} implies the elements $x_1, x_c$ cannot both be in $H_A'$.

Then $\mathcal{P}(Y',i)$ consists of all the blocks in $\A_i$ that do not contain $x_2,\ldots,x_c$, or $z$. This implies the elements of $\mathcal{P}(Y',i)$ must contain $x_1$, since every hand of $\A_i$ necessarily intersects $Y \cup \{z\}$. Since $H_A' \in \mathcal{P}(Y',i)$, we have that this set is nonempty, so $Y'$ is a possible hand for Cathy. But if Cathy holds $Y'$, Alice must hold $x_1$, which contradicts the security assumption.
\end{proof}

We remark that it is possible to use the above results to show that any hand is actually possible for Cathy in this case. That is, consider an $(a,b,c)$-strategy such that $a-c=2$, which is informative for Bob and weakly 1-secure against Cathy. Given an announcement $\A_i$ and a hand $H_C \in\binom X c$, we can show $\mathcal{P}(H_C ,i) \neq \emptyset$. We do not include the proof here, however, as we do not require this strong of a result for our purposes.

We are now ready to show one of our main results concerning the special case $c=a-2$. Namely, any $(a, b, a-2)$-strategy that is informative, equitable, and perfectly 1-secure also satisfies $c=1$:

\begin{theorem}
\label{main.thm}
Consider an $(a,b,c)$-deal such that $a-c=2$. Suppose that Alice's strategy is equitable, informative for Bob, and perfectly 1-secure against Cathy. Then $a=3$ and hence $c=1$.
\end{theorem}

\begin{proof}
Consider an announcement $\A_i$. Suppose $H_A = \{x_1,\ldots,x_a\} \in \A_i$. First note that $a \geq 3$, since $c \geq 1$ and $a-c =2$.

Let $B_{x_1}$ be the block neighborhood of $x_1$ and suppose $|B_{x_1}| = r$. Let $B_{x_2}$ be the block neighborhood of $x_2$ and suppose $|B_{x_2}| = s$. By Lemma~\ref{min4.lemma}, we have $r,s \geq 2$.

Set $H_C = \{x_3,\ldots,x_a\}$. By Lemma~\ref{poss.hand1}, we have $\mathcal{P}(H_C, i) \neq \emptyset$. Then we have (by Lemma~\ref{info.lemma}), $|B_{x_1} \backslash B_{H_C}| = r-1$ and $|B_{x_2} \backslash B_{H_C}| = s-1$. By Theorem~\ref{equitable}, we have $|B_{x_1} \backslash B_{H_C}| = |B_{x_2} \backslash B_{H_C}|$, so $r=s$.

Now consider $z$ such that $z \in N(x_1)$. We show $z \in N(x_2)$ as well. For the case $c=1$, Lemma~\ref{poss.hand1} implies we may set $H_C = \{z\}$. Then if $z \notin N(x_2)$, by Lemma~\ref{info.lemma}, we would have $|B_{x_1} \backslash B_{H_C}| = r-1$ and $|B_{x_2} \backslash B_{H_C}| = s$. By Theorem~\ref{equitable}, we have $r-1 = s$, a contradiction since $r=s$. It now suffices to consider $a \geq 4$. Set $H_C = \{x_4,\ldots,x_a,z\}$; by Lemma~\ref{poss.hand2}, we have $\mathcal{P}(H_C,i) \neq \emptyset$, so $H_C$ is a possible hand for Cathy. If $z \notin N(x_2)$, by Lemma~\ref{info.lemma}, we would have $|B_{x_1} \backslash B_{H_C}| = r-2$ and $|B_{x_2} \backslash B_{H_C}| = s-1$. But by Theorem~\ref{equitable}, we have $r-2 = s-1$, a contradiction since $r=s$. Therefore $z \in N(x_2)$.

Suppose $H_A'= \{z_1,\ldots,z_{a-1}, x_1\} \in B_{x_1}$, where $H_A' \neq H_A$, and set $H_C = \{z_1,\ldots, z_{a-2}\}$. By Lemma~\ref{info.lemma}, $|B_{x_1} \backslash B_{H_C}| = r-1$. By the above argument, $z_1,\ldots,z_{a-2} \in N(x_2)$ and by Lemma~\ref{info.lemma}, these points occur in different blocks of $x_2$ and each point occurs exactly once. So $|B_{x_2} \backslash B_{H_C}| = s-(a-2)$. By Lemma~\ref{poss.hand1}, we have $\mathcal{P}(H_C, i) \neq \emptyset$. So by Theorem~\ref{equitable}, we have $|B_{x_1} \backslash B_{H_C}|=|B_{x_2} \backslash B_{H_C}|$, so $r-1 = s-a+2$. Since we also have $r=s$, this implies $a=3$, as desired.
\end{proof}

In light of Theorem~\ref{main.thm}, we now focus on $(3,n-4,1)$-strategies that are equitable and perfectly 1-secure. Given this special case, the stronger security assumption allows us to state some useful results concerning the neighborhoods of particular cards. We first show that any two points must have a common neighbor, which provides the basis for a much stronger result concerning neighborhoods. In fact, the neighborhoods of any two distinct points (minus the points themselves) are the same. The next two lemmas are the final ingredients needed for our second main result, namely that announcements in such strategies are necessarily \emph {Steiner triple systems}.

\begin{lemma}\label{neighbor}Suppose $(a,b,c) = (3,n-4,1)$ and fix a corresponding announcement $\A_i$. Suppose that Alice's strategy is equitable, informative for Bob, and perfectly 1-secure against Cathy. Then for any distinct $x,y \in X$, there exists $z \in X$ such that $z \in N(x) \cap N(y)$.
\end{lemma}

\begin{proof} We proceed by contradiction. Let $x,y \in X$ and suppose $N(x) \cap N(y) = \emptyset$. We proceed by using a combination of Lemma~\ref{info.lemma} and the results of Theorem~\ref{equitable} to count and compare the size of the block neighborhoods of $x$ and $y$ in light of possible hands for Cathy. Recall that, from Cathy's point of view, the block neighborhoods of $x$ and $y$ must have the same size.

Let $B_x$ be the block neighborhood of $x$ and suppose $|B_x| = r$. Let $B_y$ be the block neighborhood of $y$ and suppose $|B_y| = s$. From Lemma~\ref{min4.lemma}, we have $r, s \geq 2$. 

Thus, there must be some $\ell \in N(x)$ such that $\ell \notin N(y)$. By Lemma~\ref{poss.hand1}, we may set $H_C = \{\ell\}$. Consider $B_x \backslash B_{H_C}$. By Lemma~\ref{info.lemma} and since $\ell \in N(x)$, we see that $|B_x \backslash B_{H_C}| = |B_x| - 1 = r-1$. Since $\ell \notin N(y)$, we have $|B_y \backslash B_{H_C}| = s$. Then since Alice's strategy is perfectly 1-secure against Cathy, by Theorem~\ref{equitable}, we also have $|B_x \backslash B_{H_C}| = |B_y \backslash B_{H_C}|$. This implies $s = r-1$.

Similarly, we have some $\ell' \in N(y)$ such that $\ell' \notin N(x)$. By Lemma~\ref{poss.hand1}, we may set $H_C = \{\ell'\}$. By the same argument as above, we have $|B_y \backslash B_{H_C}| = |B_y| - 1 = s-1$ and $|B_x \backslash B_{H_C}| = r$. Since $|B_x \backslash B_{H_C}| = |B_y \backslash B_{H_C}|$, we conclude $s = r+1$.

Thus we have a contradiction.
\end{proof}

\begin{lemma}\label{neighborhood equivalence}Suppose $(a,b,c) = (3,n-4,1)$ and and fix a corresponding announcement $\A_i$. Suppose that Alice's strategy is equitable, informative for Bob, and perfectly 1-secure against Cathy. Then for any distinct $x,y \in X$, we have $N(x)\backslash \{y\} = N(y) \backslash \{x\}$.
\end{lemma}

\begin{proof}
Let $x,y \in X$ be distinct. Let $B_x$ be the block neighborhood of $x$ and suppose $|B_x| = r$. Let $B_y$ be the block neighborhood of $y$ and suppose $|B_y| = s$. From Lemma~\ref{min4.lemma}, we have $r, s \geq 2$. We use a technique similar to that used in the proof of Lemma~\ref{neighbor}, i.e., counting and comparing the sizes of block neighborhoods.

By Lemma~\ref{neighbor}, there exists $z \in X$ such that $z \in N(x) \cap N(y)$. We observe that, by Lemma~\ref{poss.hand1}, we may set $H_C = \{z\}$. Then we have (by Lemma~\ref{info.lemma}) $|B_x \backslash B_{H_C}| = r-1$ and $|B_y \backslash B_{H_C}| = s-1$. By Theorem~\ref{equitable}, we have $|B_x \backslash B_{H_C}| = |B_y \backslash B_{H_C}|$, so $r=s$.

We proceed by contradiction. First suppose there is $\ell \neq x,y$ such that $\ell \in N(x)$ but $\ell \notin N(y)$. By Lemma~\ref{poss.hand1}, we may set $H_C = \{\ell\}$.  We then have (by Lemma~\ref{info.lemma}) $|B_x \backslash B_{H_C}| = r-1$ and $|B_y \backslash B_{H_C}| = s$. By Theorem~\ref{equitable}, we have $|B_x \backslash B_{H_C}| = |B_y \backslash B_{H_C}|$, so $s = r-1$, a contradiction. This implies that $N(x)\backslash \{y\} \subseteq N(y) \backslash \{x\}$.

Now suppose there there is $\ell' \neq x,y$ such that $\ell' \in N(y)$ but $\ell' \notin N(x)$. We observe that, by Lemma~\ref{poss.hand1}, we may set $H_C = \{\ell'\}$.  We then have (by Lemma~\ref{info.lemma}) $|B_x \backslash B_{H_C}| = r$ and $|B_y \backslash B_{H_C}| = s-1$. By Theorem~\ref{equitable}, we have $|B_x \backslash B_{H_C}| = |B_y \backslash B_{H_C}|$, so $r = s-1$, a contradiction. This implies that $N(y) \backslash \{x\} \subseteq N(x)\backslash \{y\}$.
\end{proof}

\begin{theorem}\label{Steiner}Suppose $(a,b,c) = (3,n-4,1)$ and suppose that Alice's strategy is equitable, informative for Bob, and perfectly 1-secure against Cathy. Then every announcement is a Steiner triple system.
\end{theorem}

\begin{proof}
Fix an $(a,b,c)$-deal and suppose Alice's strategy is equitable, informative for Bob, and perfectly 1-secure against Cathy. Consider a corresponding announcement $\A_i$. Then in particular, each hand of $\A_i$ has size $3$.

We first observe that Lemma~\ref{info.lemma} implies that any pair $x,y \in X$ occurs in at most one hand of $\A_i$. It remains to show that any pair $x,y \in X$ occurs in exactly one hand of $\A_i$.

Let $x,y \in X$. By Lemma~\ref{neighbor}, there is some point $z \in X$ such that $z \in N(x) \cap N(y)$. In particular, $x \in N(z)$. By Lemma~\ref{neighborhood equivalence}, we have $N(z)\backslash \{y\} = N(y) \backslash \{z\}$. Since $z \in N(y)$, we see that $N(z)\backslash \{y\} \cup \{z\} = N(y)$. But $x \neq y,z$ and $x \in N(z)$, so we have $x \in N(y)$. This gives us the desired result. 
\end{proof}

We present an interesting example in the case $a=4$, $c=2$. 

\begin{example}
It was proven by Chouinard \cite{Ch}
that there is a large set of $2$-$(13,4,1)$-designs. There are
$\binom{11}{2} = 55$ designs in the large set. This yields a deterministic 
$(4,7,2)$-strategy that is informative for Bob. We can easily determine the security of the scheme
against Cathy. Suppose that Alice's announcement is $\A_i$ and Cathy's hand is $H_C = \{y,z\}$.
There is a unique block in $\A_i$ that contains the pair $\{y,z\}$, say $\{w,x,y,z\}$.
There are three blocks that contain $y$ but not $z$, and three blocks that contain
$z$ but not $y$. Since $\A_i$ contains $13$ blocks, it follows that the set $\mathcal{P}(\{y,z\},i)$
consists of six blocks. Within these six blocks, $w$ and $x$ occur three times, and 
every point in $X \backslash  \{w,x,y,z\}$ occurs twice. Therefore, we have
\[ \Prob [ w \in H_A | H_C ] = \Prob [  x \in H_A | H_C ] = \frac{1}{2}\]
and \[ \Prob [ u \in H_A | H_C ]  = \frac{1}{3}\]
for all $u \in X \backslash \{w,x,y,z\}$.
If a  $(4,7,2)$-strategy were perfectly $1$-secure against Cathy (which is impossible, in view
of Theorem \ref{main.thm}), we would have $\Prob [ u \in H_A | H_C ]  = {4}/{11}$
for all $u \in X \backslash H_C$.
\end{example}

\section{Generalized Notions of Security}
\label{sec: Generalized Notions of Security}

We may generalize the definitions of weak and perfect 1-security to weak and perfect $\delta$-security in the natural way.

\begin{definition}
\mbox{\quad}
Let $1 \leq \delta \leq a$.
\begin{enumerate}
\item Alice's strategy is {\em weakly $\delta$-secure against Cathy} provided that for any $\delta'$ such that $1 \leq \delta' \leq \delta$,
for any announcement $i$, 
for any $H_C \in \binom{X}{c}$ such that $\mathcal{P}(H_C,i) \neq \emptyset$, 
and for any $x_1,\ldots, x_{\delta'} \in X \backslash H_C$, it holds that 
\[0 < \Prob [ x_1,\ldots,x_{\delta'} \in H_A | i,H_C] < 1.\]
Weak security means that, from Cathy's point of view, any set of $\delta$ or fewer elements from $X \backslash H_C$
may or may not be held by Alice.

\item Alice's strategy is {\em perfectly $\delta$-secure against Cathy} provided that for any $\delta'$ such that $1 \leq \delta' \leq \delta$,
for any announcement $i$, 
for any $H_C \in \binom{X}{c}$ such that $\mathcal{P}(H_C,i) \neq \emptyset$, 
and for any $x_1,\ldots,x_{\delta'} \in X \backslash H_C$, it holds that 
\[ \Prob [ x_1,\ldots,x_{\delta'} \in H_A | i,H_C] = \frac{\binom{a}{\delta'}}{\binom{a+b}{\delta'}}.\]
Perfect security means that, from Cathy's point of view, the probability that 
any set of $\delta$ or fewer cards from $X \backslash H_C$
is held by Alice is a constant.
\end{enumerate}
It is obvious that perfect $\delta$-security implies weak $\delta$-security.
\end{definition}

\noindent{\bf Remark:}
The condition $\mathcal{P}(H_C,i) \neq \emptyset$ is included to account for the possibility
that an announcement $i$ is not compatible with certain hands $H_C$ held
by Cathy. 

\medskip

The conditions for weak and perfect $\delta$-security depend on the probability distributions $p_{H_A}$
and the possible announcements. As before, we will derive simpler, but equivalent, conditions 
of a combinatorial nature 
when Alice's strategy is equitable. 

\begin{theorem}\label{equitable1}
Suppose that Alice's strategy is $\gamma$-equitable. Then the following hold:
\begin{enumerate}
\item \label{weak1} Alice's strategy is weakly $\delta$-secure against Cathy if and only if, for any $\delta'$ such that $1 \leq \delta' \leq \delta$,
for any announcement $i$, 
for any $H_C \in \binom{X}{c}$ such that $\mathcal{P}(H_C,i) \neq \emptyset$, 
and for any $x_1,\ldots,x_{\delta'} \in X \backslash H_C$, it holds that
\[ 1 \leq | \{ H_A \in \mathcal{P}(H_C,i) : x_1,\ldots,x_{\delta'} \in H_A \} | \leq |\mathcal{P}(H_C,i))|  - 1.\]

\item \label{strong equitable1}
Alice's strategy is perfectly $\delta$-secure against Cathy if and only if, for any announcement $i$ and
for any $H_C \in \binom{X}{c}$ such that $\mathcal{P}(H_C,i) \neq \emptyset$, 
it holds that
\[ | \{ H_A \in \mathcal{P}(H_C,i) : x_1,\ldots,x_{\delta} \in H_A \}| = 
\frac{\binom{a}{\delta} \, |\mathcal{P}(H_C,i)| }{ \binom{a+b}{\delta}}\]
for any $x_1,\ldots,x_{\delta} \in X \backslash H_C$.
\end{enumerate}
\end{theorem}

\begin{proof}
Let $1 \leq \delta' \leq \delta$.

Since (\ref{equitable.eq}) (from Lemma~\ref{equitable.lem}) holds, it immediately follows that 
\begin{equation}\label{single card probability1}
\Prob [ x_1,\ldots,x_{\delta'} \in H_A | i,H_C] = \frac{|\{ H_A \in \mathcal{P}(H_C,i) : x_1,\ldots,x_{\delta'} \in H_A \}|}{ |\mathcal{P}(H_C,i)|}.
\end{equation}
Using Equation~\eqref{single card probability1}, we observe that
\begin{equation*}
0 < \frac{|\{ H_A \in \mathcal{P}(H_C,i) : x_1,\ldots,x_{\delta'} \in H_A \}|}{ |\mathcal{P}(H_C,i)|}  < 1
\end{equation*}
holds if and only if 
\begin{equation*} 1 \leq |\{ H_A \in \mathcal{P}(H_C,i) : x_1,\ldots,x_{\delta'} \in H_A \}|
\leq |\mathcal{P}(H_C,i)| - 1.
\end{equation*}
This gives the first condition of the theorem.

For the second condition of the theorem, we first remark that, if the given security property holds for $\delta$, it will automatically hold for $\delta'$ such that $1 \leq \delta' \leq \delta$. This is because the security property for $\delta$ says that every $\delta$-subset occurs the same number of times within a certain set of blocks of size $|\mathcal{P}(H_C,i)|$. That is, we have a $t$-design with $t = \delta$. It is a standard result that every $t$-design is a $t'$-design for all $t' \leq t$. Thus it suffices to show that, for any announcement $i$ and
for any $H_C \in \binom{X}{c}$ such that $\mathcal{P}(H_C,i) \neq \emptyset$, and for any $x_1,\ldots,x_{\delta} \in X \backslash H_C$,
then
\[ | \{ H_A \in \mathcal{P}(H_C,i) : x_1,\ldots,x_{\delta} \in H_A \}| = 
\frac{\binom{a}{\delta} \, |\mathcal{P}(H_C,i)| }{ \binom{a+b}{\delta}}\]
holds if and only if \[ \Prob [ x_1,\ldots,x_{\delta} \in H_A | i,H_C] = \frac{\binom{a}{\delta}}{\binom{a+b}{\delta}}.\]

Define $r_{x_1,\ldots,x_\delta} = |\{ H_A \in \mathcal{P}(H_C,i) : x_1,\ldots,x_\delta \in H_A \}|$.
Alice's strategy is perfectly $\delta$-secure against Cathy if and only if 
the value $\Prob [ x_1,\ldots,x_\delta \in H_A | i,H_C]$ is independent of the $\delta$-subset $\{x_1,\ldots,x_\delta\}$. From (\ref{single card probability1}), 
this  occurs if and only if
$r_{x_1,\ldots,x_\delta}$ is independent of the $\delta$-subset $\{x_1,\ldots,x_\delta\}$.
We have that 
\[
\sum_{D \in \binom{X \backslash H_C}{\delta}} r_D = \binom{a}{\delta} \, |\mathcal{P}(H_C,i)|.
\]
There are $\binom{a+b}{\delta}$ terms  $r_D$ in the above sum. These terms are all equal if and only if they all have the
value $r = \binom{a}{\delta} \, |\mathcal{P}(H_C,i)| / \binom{a+b}{\delta}$.  This completes the proof.
\end{proof}

\begin{lemma}
\label{t-design}
Suppose that each announcement in an equitable $(a,b,1)$-strategy is a $t$-$(n,a,\lambda)$-design. Then the strategy is perfectly $(t-1)$-secure against Cathy.
\end{lemma}

\begin{proof}
Given an announcement $\A_i$  and a point $x$, there are 
\[\lambda \left( \frac{\binom{n}{t}}{\binom{a}{t}} - \frac{\binom{n-1}{t-1}}{\binom{a-1}{t-1}} \right)\] 
blocks in 
$\A_i$ that do not contain $x$. For any subset $S=\{x_1,\ldots, x_{s-1}\} \subset X \backslash \{x\}$ of size $s-1$, where $1 \leq s \leq t$, the subset $S$
is contained in precisely 
\[ \lambda \left( \frac{\binom{n-s+1}{t-s+1}}{\binom{a-s+1}{t-{s-1}}} - \frac{\binom{n-s}{a-s}}{\binom{a-s}{t-s}} \right)\] of these blocks.
\end{proof}

\begin{remark} Lemma~\ref{t-design} is a generalization of Lemma~\ref{2-design}.
\end{remark}

Lemma~\ref{t-design} immediately implies the following:
\begin{corollary}The construction method given in Theorem~\ref{aut.thm}, which shows how to obtain an equitable strategy from a single starting $t-(n,a,1)$-design, where $t=a-1$,  yields a strategy that is perfectly $(a-2)$-secure.
\end{corollary}

\subsection{Strategies with $c= a-d$}
\label{subsec: Strategies with $c=a-d$}
In this section, we generalize the results of Section~\ref{subsec: Strategies with $c=a-2$}. That is, we consider the case of $(a, b, a-d)$-deals that are simultaneously informative for Bob and perfectly $(d-1)$-secure against Cathy. Where possible, we weaken our assumption that the strategy is equitable and satisfies perfect $(d-1)$-security to achieve the given result. We do assume that the strategies discussed are informative throughout.

Although the results of Section~\ref{subsec: Strategies with $c=a-2$} are subsumed by the parallel results of this section, we feel it is useful to include both. Section~\ref{subsec: Strategies with $c=a-2$} provides a good basis for understanding the results of this section; the proofs of the generalized results are much more technical and complicated than those for the simple case where $c=a-2$. For readability, we include a list of correspondences between the results of these two sections in Table~\ref{table: Correspondences between Results}.

\begin{table}[h]
\caption{ Correspondences between Results}
\label{table: Correspondences between Results}
\begin{tabular}{c|c}
Result in Section~\ref{subsec: Strategies with $c=a-2$} & Corresponding Result in Section~\ref{subsec: Strategies with $c=a-d$} \\ \hline
Lemma~\ref{info.lemma} & Lemma~\ref{info.lemma1}\\
Lemma~\ref{poss.hand1} & Lemma~\ref{poss.hand.gen}\\
Lemma~\ref{min4.lemma} & Lemma~\ref{min4.lemma.gen} \\
Lemma~\ref{poss.hand2} & Lemma~\ref{poss.hand2.gen}\\
Theorem~\ref{main.thm} & Theorem~\ref{main.thm1}\\
 Lemmas~\ref{neighbor},~\ref{neighborhood equivalence}& Lemma~\ref{neighbor.gen}\\
Theorem~\ref{Steiner} & Theorem~\ref{Steiner.gen}\\
\end{tabular}
\end{table}

The main result of this section is that any $(a, b, a-d)$-strategy that is informative, equitable, and perfectly $(d-1)$-secure also satisfies $c=1$; that is, $d=a-1$. Moreover, announcements in such strategies are necessarily $d$-$(n,d+1,1)$ designs. To achieve these results, however, we do need an additional assumption; namely that $b$ is sufficiently large. As we will see, taking $b \geq d -1$ suffices. We do not know if this assumption is necessary, however.

\begin{lemma}\label{info.lemma1}Consider an $(a,b,c)$-deal such that $a-c=d$, and a corresponding announcement $\A_i$. Suppose that Alice's strategy is informative for Bob. Then for any distinct $x_1,\ldots, x_d \in X$, there is at most one hand $H_A \in \A_i$ such that $x_1,\ldots, x_d \in H_A$. That is, $|B_{x_1,\ldots,x_d}| \leq 1$.
\end{lemma}

\begin{proof}This follows directly from Theorem~\ref{infor.thm}.
\end{proof}

\begin{lemma}\label{poss.hand.gen}Consider an $(a,b,c)$-deal such that $a-c=d$ and $b \geq d-1$. Fix a corresponding announcement $\A_i$. Suppose that Alice's strategy is informative for Bob and weakly $(d-1)$-secure against Cathy. Let $H_A \in \A_i$ and $Y \subset H_A$ satisfy $|Y| = c$. Then $\mathcal{P}(Y,i) \neq \emptyset$.
\end{lemma}


\begin{proof} 
We proceed by contradiction. Write $H_A = \{x_1,\ldots, x_a\}$ and $Y=\{x_1,\ldots,x_c\}$. Suppose $\mathcal{P}(Y,i) = \emptyset$. Then every hand of $\A_i$ intersects $Y$. 

Now, since Alice's strategy is weakly $(d-1)$-secure against Cathy, there must be some $H_A' \in \A_i$ such that $H_A' \neq H_A$. By the above argument, $|H_A' \cap Y|\geq 1$.

Suppose $|H_A' \cap Y| = \ell'$ and $|H_A \cap H_A'| = \ell$. Note that $\ell' \leq \ell \leq d-1$ holds by Lemma~\ref{info.lemma1} and $\ell' \leq c$ holds by construction. Without loss of generality, assume $H_A'$ contains $x_1,\ldots, x_{\ell'}$.

We now wish to construct a special possible hand for Cathy, say $Y'$, that will allow us to derive a contradiction. That is, we will construct a $Y'$ such that $H_A \notin \mathcal{P}(Y',i)$, but $\mathcal{P}(Y',i)$ contains a hand $H_A''$ satisfying $|H_A \cap H_A''| \geq d$. To ensure $Y'$ is a possible hand for Cathy, we construct $Y'$ using elements that do not appear in $H_A'$, so that $H_A' \in \mathcal{P}(Y',i)$ and hence $\mathcal{P}(Y',i)$ is nonempty. For technical reasons, we pick one element $z$ that occurs in $H_A$ but not $H_A'$, and $\ell'-1$ elements $z_1,\ldots, z_{\ell'-1}$ that occur outside of both $H_A$ and $H_A'$, and use these in our construction of $Y'$.

Since $H_A \neq H_A'$, there is some $z \in H_A$ such that $z \notin H_A'$ and $z$ is distinct from $x_{\ell' + 1},\ldots, x_{c}$. To see this, write $\ell = \ell' + t$ for some $t$. There are $d - t$ elements in $H_A \backslash (Y \cup H_A')$. Thus $d - t \geq 1$ suffices, but necessarily we have $t \leq \ell \leq d-1$. From a technical standpoint, we need such a point $z$ for the case $c=1$; this will ensure that $H_A \cap Y' \neq \emptyset$.

In addition, we may pick distinct $z_1,\ldots, z_{\ell'-1} \notin H_A \cup H_A'$. This follows because there are at least $a + b + (a-d) - (2a - \ell)= b -d + \ell$ points not in $H_A \cup H_A'$. We have $b - d + \ell \geq \ell'-1$ so long as $b \geq d-1$, which is true by assumption.

Define $Y' = \{x_{\ell' + 1},\ldots,x_{c}, z_1,\ldots z_{\ell'-1}, z\}$, so $|Y'| =c$. Then $H_A' \in \mathcal{P}(Y',i)$ by construction, so this set is nonempty, and therefore $Y'$ is a possible hand for Cathy. Note also that $H_A \notin \mathcal{P}(Y',i)$.

Consider the set $T= \{x_{c+1}, \ldots, x_a\} \subset H_A$, which contains $d$ elements. Note that at most one of these elements is $z$, so we may pick a subset $T' \subset T$ satisfying $|T'| = d-1$ and $T' \cap Y' = \emptyset$. Since the scheme satisfies weak $(d-1)$-security by assumption, $T' \subset H_A''$ for some $H_A'' \in \mathcal{P}(Y',i)$. Note that $H_A'' \neq H_A$. Now, $H_A''$ must intersect $Y$ (but not $Y'$), so $H_A''$ must contain an element from $\{x_1,\ldots, x_{\ell'} \}$. Suppose (without loss of generality) that $H_A''$ contains $x_1$. Then the set $\{x_1\} \cup T'$ of size $d$ appears in both $H_A$ and $H_A''$, a contradiction.
\end{proof}

\begin{lemma}\label{min4.lemma.gen} Consider an $(a,b,c)$-deal such that $a-c=d$ and $b \geq d-1$. Fix a corresponding announcement $\A_i$. Suppose that Alice's strategy is informative for Bob and weakly $(d-1)$-secure against Cathy. Suppose $D= \{x_1,\ldots,x_{d-1}\} \subset X$. Then there are at least two hands of $\A_i$ containing $D$ and at least two hands of $\A_i$ that do not contain $D$.
\end{lemma}


\begin{proof} We first show there exist $H_A, H_A' \in \A_i$ which contain $D$. Now, there must be some $H_A \in \A_i$ satisfying $D \subseteq H_A$, as otherwise $D$ is not held by Alice, which contradicts the assumption that the scheme is weakly $(d-1)$-secure. Suppose all other hands of $\A_i$ do not contain $D$. Let $Y \subset H_A$ such that $|Y| = c$. By Lemma~\ref{poss.hand.gen}, we have $\mathcal{P}(Y,i) \neq \emptyset$; that is, $Y$ is a possible hand for Cathy. But since $H_A \notin \mathcal{P}(Y,i)$, there must be some other $H_A' \in \A_i$ such that $D \subseteq H_A'$. Otherwise, if Cathy holds $Y$, then Cathy knows $D$ is not held by Alice.

We show in a similar fashion there exist $H_A, H_A' \in \A_i$ which do not contain $D$. There must be some $H_A \in \A_i$ satisfying $D \nsubseteq H_A$, as otherwise $D$ must be held by Alice. Suppose all other hands of $\A_i$ contain $D$. Let $Y \subset H_A$ such that $|Y| = c$. By Lemma~\ref{poss.hand.gen}, we have $\mathcal{P}(Y,i) \neq \emptyset$; that is, $Y$ is a possible hand for Cathy. But since $H_A \nsubseteq \mathcal{P}(Y,i)$, there must be some other $H_A' \in \A_i$ such that $D \notin H_A'$. Otherwise, if Cathy holds $Y$, then Cathy knows $D$ is held by Alice.
\end{proof}

\begin{lemma}\label{poss.hand2.gen} Consider an $(a,b,c)$-deal such that $a-c=d$ and $b \geq d-1$. Fix a corresponding announcement $\A_i$. Suppose that Alice's strategy is informative for Bob and weakly $(d-1)$-secure against Cathy. Let $H_A \in \A_i$ and $Y \subset H_A$ satisfy $|Y| = c-1$. Let $z \in X$ such that $z \notin H_A$. Then $\mathcal{P}(Y \cup \{z\},i) \neq \emptyset$.
\end{lemma}


\begin{proof} 
This result is only interesting for $c \geq 2$. The case $c=1$ is follows directly from Lemma~\ref{poss.hand.gen}, since for any $z \in X$, there is some hand in $\A_i$ that contains $z$.

We proceed by contradiction. Write $H_A = \{x_1,\ldots, x_a\}$ and $Y=\{x_1,\ldots,x_{c-1}\}$. Let $z \in X$ such that $z \notin H_A$. Suppose $\mathcal{P}(Y \cup \{z\},i) = \emptyset$. Then every hand of $\A_i$ intersects $Y \cup \{z\}$. 

Now, since Alice's strategy is weakly $(d-1)$-secure against Cathy, there must be some $H_A' \in \A_i$ such that $H_A' \neq H_A$ and $z \notin H_A'$ (Lemma~\ref{min4.lemma.gen} gives a stronger result). By the above argument, $|H_A' \cap Y|\geq 1$.

Suppose $|H_A' \cap Y| = \ell'$ and $|H_A \cap H_A'| = \ell$. Note that $\ell' \leq \ell \leq d-1$ by Lemma~\ref{info.lemma1} and $\ell' \leq c-1$ holds by construction. Without loss of generality, assume $H_A'$ contains $x_1,\ldots, x_{\ell'}$.

We now wish to construct a special possible hand for Cathy, say $Y'$, that will allow us to derive a contradiction. That is, we will construct a $Y'$ such that $H_A \notin \mathcal{P}(Y',i)$, but $\mathcal{P}(Y',i)$ contains a hand $H_A''$ satisfying $|H_A \cap H_A''| \geq d$. To ensure $Y'$ is a possible hand for Cathy, we construct $Y'$ using elements that do not appear in $H_A'$, so that $H_A' \in \mathcal{P}(Y',i)$ and hence $\mathcal{P}(Y',i)$ is nonempty. For technical reasons, we pick one element $z'$ that occurs in $H_A$ but not $H_A'$, and $\ell'-2$ elements $z_1,\ldots, z_{\ell'-2}$ that occur outside of both $H_A$ and $H_A'$, and use these in our construction of $Y'$.

Since $H_A \neq H_A'$, there is some $z' \in H_A$ such that $z' \notin H_A'$ and $z'$ is distinct from $x_{\ell' + 1},\ldots, x_{c}$. To see this, write $\ell = \ell' + t$ for some $t$. There are $d - t$ elements in $H_A \backslash (Y \cup H_A' \cup \{x_{c}\})$. Thus $d-t \geq 1$ suffices, which holds because $t \leq \ell \leq d-1$ by the security assumption. From a technical standpoint, we need such a point $z'$ for the case $c=2$; this will ensure that $H_A \cap Y' \neq \emptyset$.

In addition, we may pick distinct $z_1,\ldots, z_{\ell'-2} \notin H_A \cup H_A'$ such that $z_i \neq z$ for $1 \leq i \leq \ell'-2$. This follows because there are at least $a + b + (a-d) - (2a - \ell)= b -d + \ell$ points not in $H_A \cup H_A'$. We have $b - d + \ell \geq \ell'-1$ so long as $b \geq d-1$, which holds by assumption. (Note that we need $\ell' - 1$ points, not $\ell' - 2$ points, because $z$ must be distinct from $z_1,\ldots, z_{\ell'-2}$, and all are points occurring outside of $H_A \cup H_A'$.)

Define $Y' = \{x_{\ell' + 1},\ldots,x_{c}, z_1,\ldots z_{\ell'-2}, z, z'\}$, so $|Y'| =c$. Then $H_A' \in \mathcal{P}(Y',i)$ by construction, so this set is nonempty, and therefore $Y'$ is a possible hand for Cathy. Note also that $H_A \notin \mathcal{P}(Y',i)$.

Consider the set $T= \{x_{c+1}, \ldots, x_a\} \subset H_A$, which contains $d$ elements. Note that at most one of these elements is $z'$, so we may pick a subset $T' \subset T$ satisfying $|T'| = d-1$ and $T' \cap Y' = \emptyset$. Since the scheme satisfies $(d-1)$-weak security by assumption, $T' \subset H_A''$ for some $H_A'' \in \mathcal{P}(Y',i)$. Note that $H_A'' \neq  H_A$. But $H_A''$ must intersect $Y$ (but not $Y'$) so $H_A''$ must contain an element from $\{x_1,\ldots, x_{\ell'}\}$. Suppose (without loss of generality) $H_A''$ contains $x_1$. Then the set $\{x_1\} \cup T'$ of size $d$ appears in both $H_A$ and $H_A''$, a contradiction.
\end{proof}

\begin{theorem}
\label{main.thm1}
Consider an $(a,b,c)$-deal such that $a-c=d$ and $b \geq d-1$. Suppose that Alice's strategy is equitable, informative for Bob, and perfectly $(d-1)$-secure against Cathy. Then $a=d+1$ and hence $c=1$.
\end{theorem}


\begin{proof}
We remark that $a \geq d+1$, since $c \geq 1$.

Consider an announcement $\A_i$. Suppose $H_A = \{x_1,\ldots,x_a\} \in \A_i$. Let $B_{x_1,\ldots,x_{d-1}}$ be the block neighborhood of $x_1,\ldots,x_{d-1}$ and suppose $|B_{x_1,\ldots,x_{d-1}}| = r$. Let $B_{x_2,\ldots, x_d}$ be the block neighborhood of $x_2,\ldots, x_d$ and suppose $|B_{x_2,\ldots, x_d}| = s$. By Lemma~\ref{min4.lemma.gen}, we have $r,s \geq 2$.

Set $H_C = \{x_{d+1},\ldots,x_a\}$. By Lemma~\ref{poss.hand.gen}, we have $\mathcal{P}(H_C, i) \neq \emptyset$, so $H_C$ is a possible hand for Cathy. Then we have, by Lemma~\ref{info.lemma1}, $|B_{x_1,\ldots,x_{d-1}} \backslash B_{H_C}| = r-1$ and $|B_{x_2,\ldots, x_d} \backslash B_{H_C}| = s-1$. By Theorem~\ref{equitable1}, we have $|B_{x_1,\ldots,x_{d-1}} \backslash B_{H_C}| = |B_{x_2,\ldots, x_d} \backslash B_{H_C}|$, so we conclude $r=s$.

Now consider $z$ such that $z \in N(x_1,\ldots,x_{d-1})$, but $z \notin H_A$. We show $z \in N(x_2,\ldots,x_{d})$ as well. For the case $c=1$, Lemma~\ref{poss.hand.gen} implies we may set $H_C = \{z\}$. Then if $z \notin N(x_2,\ldots,x_{d})$, by Lemma~\ref{info.lemma1}, we would have $|B_{x_1,\ldots,x_{d-1}} \backslash B_{H_C}| = r-1$ and $|B_{x_2,\ldots, x_d} \backslash B_{H_C}| = s$. By Theorem~\ref{equitable}, we have $r-1 = s$, a contradiction since $r=s$. For $c > 1$, we may set $H_C = \{x_{d+1},\ldots,x_{a-1},z\}$. By Lemma~\ref{poss.hand2.gen}, we have $\mathcal{P}(H_C,i) \neq \emptyset$, so $H_C$ is a possible hand for Cathy. If $z \notin N(x_2,\ldots,x_{d})$, by Lemma~\ref{info.lemma1}, we would have $|B_{x_1,\ldots,x_{d-1}} \backslash B_{H_C}| = r-2$ and $|B_{x_2,\ldots, x_d} \backslash B_{H_C}| = s-1$. But by Theorem~\ref{equitable}, we have $r-2 = s-1$, a contradiction since $r=s$. Therefore $z \in N(x_2,\ldots,x_d)$.

Suppose $H_A'= \{x_1,\ldots,x_{d-1}\} \cup \{z_1,\ldots,z_{a-d+1}\} \in B_{x_1,\ldots,x_{d-1}}$ and set $H_C = \{z_1,\ldots, z_{a-d}\}$. By Lemma~\ref{poss.hand.gen}, we have $\mathcal{P}(H_C, i) \neq \emptyset$, so $H_C$ is a possible hand for Cathy. 

Now, we may pick $H_A' \neq H_A$ by Lemma~\ref{min4.lemma.gen}. Then by Lemma~\ref{info.lemma1}, we have $|B_{x_1,\ldots, x_{d-1}} \backslash B_{H_C}| = r-1$. By the above argument, $z_1,\ldots,z_{a-d} \in N(x_2,\ldots, x_d)$ and by Lemma~\ref{info.lemma1}, these points occur in different blocks of $B_{x_2,\ldots, x_d}$ and each point occurs exactly once. So $|B_{x_2,\ldots,x_d} \backslash B_{H_C}| = s-(a-d)$. So by Theorem~\ref{equitable1}, we have $|B_{x_1,\ldots,x_{d-1}} \backslash B_{H_C}|=|B_{x_2,\ldots,x_d} \backslash B_{H_C}|$, so $r-1 = s-a+d$. Since we also have $r=s$, this implies $a=d+1$, as desired.
\end{proof}

As before, we can now focus our attention on $(d+1,b,1)$-strategies that are equitable, informative for Bob, and perfectly $(d-1)$-secure.

\begin{lemma}\label{neighbor.gen}Consider a $(d+1,b,1)$-deal satisfying $b \geq d-1$ and fix a corresponding announcement $\A_i$. Suppose that Alice's strategy is equitable, informative for Bob, and perfectly $(d-1)$-secure against Cathy. Then for any distinct $D, D' \subset X$ satisfying $|D| = |D'| = d-1$, there exists $z \in X$ such that $z \in N(D) \cap N(D')$. Moreover, $N(D)\backslash \{D'\} = N(D') \backslash \{D\}$.
\end{lemma}


\begin{proof}
Let $D, D' \subset X$ satisfy $|D| = |D'| = d-1$. Let $B_D$ be the block neighborhood of $D$ and suppose $|B_{D}| = r$. Let $B_{D'}$ be the block neighborhood of $D'$ and suppose $|B_{D'}| = s$. From Lemma~\ref{min4.lemma.gen}, we have $r, s \geq 2$. 

As in the proof of Theorem~\ref{main.thm1}, suppose there exists $\ell \in N(D)$ such that $\ell \notin N(D')$. We may, by Lemma~\ref{poss.hand.gen}, set $H_C = \{\ell\}$. Then by Lemma~\ref{info.lemma1}, we see that $|B_D \backslash B_{H_C}| = r-1$ and $|B_{D'} \backslash B_{H_C}| = s$. Then since Alice's strategy is perfectly $(d-1)$-secure against Cathy, by Theorem~\ref{equitable1}, we also have $|B_D \backslash B_{H_C}| = |B_{D'} \backslash B_{H_C}|$. This implies $s = r-1$.

Similarly, suppose we have some $\ell' \in N(D')$ such that $\ell' \notin N(D)$. By Lemma~\ref{poss.hand.gen}, we may set $H_C = \{\ell'\}$. By the same argument as above, we have $|B_{D'} \backslash B_{H_C}| = s-1$ and $|B_{D} \backslash B_{H_C}| = r$. Since $|B_D \backslash B_{H_C}| = |B_{D'} \backslash B_{H_C}|$, we conclude $s = r+1$.

The above argument implies a contradiction if there exists both $\ell \in N(D)$ such that $\ell \notin N(D')$ and $\ell' \in N(D')$ such that $\ell' \notin N(D)$. But if $N(D) \cap N(D') = \emptyset$, such an $\ell$ and $\ell'$ must exist, since $r,s \geq 2$. Thus, we conclude that there exists $z \in N(D) \cap N(D')$.

Moreover, by Lemma~\ref{poss.hand.gen}, we may set $H_C = \{z\}$. Then by Lemma~\ref{info.lemma1}, we see that $|B_D \backslash B_{H_C}| = r-1$ and $|B_y \backslash B_{H_C}| = s-1$. Then since Alice's strategy is perfectly $(d-1)$-secure against Cathy, by Theorem~\ref{equitable1}, we also have $|B_D \backslash B_{H_C}| = |B_{D'} \backslash B_{H_C}|$. This implies $r = s$. But this also implies that there cannot be $\ell \in N(D)$ such that $\ell \notin N(D')$ or $\ell' \in N(D')$ such that $\ell' \notin N(D)$. Thus we conclude $N(D)\backslash \{D'\} = N(D') \backslash \{D\}$, as desired.
\end{proof}

\begin{theorem}
\label{Steiner.gen}Suppose $(a,b,c) = (d+1,n-(d+2),1)$ satisfying $b \geq d-1$ and suppose that Alice's strategy is equitable, informative for Bob, and perfectly $(d-1)$-secure against Cathy. Then every announcement is a $d-(n,d+1,1)$-design. 
\end{theorem}


\begin{proof}
Fix an $(d+1,n-(d+2),1)$-deal and suppose Alice's strategy is equitable, informative for Bob, and perfectly $(d-1)$-secure against Cathy. Consider a corresponding announcement $\A_i$. Then in particular, each hand of $\A_i$ has order $d+1$.

We first observe that Lemma~\ref{info.lemma1} implies that any set of $d$ elements of $X$ occurs in at most one hand of $\A_i$. It remains to show that $D = \{x_1,\ldots, x_d\}$ occurs in exactly one hand of $\A_i$ for any $x_1,\ldots, x_d \in X$. 

Write $D' = \{x_1,\ldots, x_{d-1}\}$ and $D'' = \{x_2,\ldots, x_{d}\}$. By Lemma~\ref{neighbor.gen}, there exists $z \in X$ satisfying $z \in N(D') \cap N(D'')$. (Note that this implies $z \neq x_1,\ldots, x_d$.) That is, there exist hands $H_A, H_A' \in \A_i$ satisfying $D' \cup \{z\} \subseteq H_A$ and $D'' \cup \{z\} \subseteq H_A'$.

Also by Lemma~\ref{neighbor.gen}, we have \[N(D'') \backslash \{z, x_2,\ldots,x_{d-1}\} = N(z, x_2,\ldots,x_{d-1}) \backslash D''.\] This is equivalent to \[N(D'') \backslash \{z\} = N(z, x_2,\ldots,x_{d-1}) \backslash \{x_d\}.\]

Now $D' \cup \{z\} \subseteq H_A$ implies $x_1 \in N(z, x_2,\ldots,x_{d-1})$. Given $x_1 \neq z, x_d$, we conclude  $x_1 \in N(D'')$. That is, $D = D'' \cup \{x_1\}$ occurs in some hand of $\A_i$, as desired.
\end{proof}

\begin{example}The construction given in Example~\ref{perfectly 2-secure} is actually an example of a 2-equitable $(4,3,1)$-strategy that is informative for Bob and perfectly 2-secure against Cathy. The fact that the scheme is perfectly 2-secure follows from Lemma~\ref{t-design}.
\end{example}

\section{Discussion and Comparison with Related Work}
\label{sec: Related Work}

As mentioned in Section~\ref{sec: Introduction}, there are have been many papers studying the Russian cards problem and generalizations of it. Here we concentrate on recent work that takes a combinatorial approach~\cite{AAADH05,AD09,ACDFJS11,CDFJS}.

Albert et al.~\cite{AAADH05} consider the card problem from both epistemic logic and combinatorial perspectives, establishing axioms CA1, CA2, and CA3 that are roughly equivalent to our requirements for a protocol to be informative and weakly 1-secure in the $\gamma$-equitable case. The difference is that the authors~\cite{AAADH05} treat security on the announcement level; that is, they identify various announcements as \emph{good} if the relevant properties hold for any possible hand for Alice \emph{in the given announcement}. No assumption is made that, for every possible hand for Alice, an announcement is defined, or that a good announcement even exists. Our definitions, on the other hand, require that Alice have a (secure) announcement for every possible hand $H_A \in \binom{X}{a}$. In particular, we argue that it is not possible to formally define or discuss the security of a scheme using definitions that focus on individual announcements.

The authors~\cite{AAADH05} present several useful results, some of which we have cited in this paper, on the relationships between the parameters $a$ and $c$, and $b$ and $c$, as well as bounds on the minimum and maximum number of hands in a good announcement. The focus is on the level of announcements throughout: the authors argue that, to minimize information gained by Cathy, the size of the announcement should be maximized. Moreover, the authors show good announcements exist for some special cases, including using block designs for the case $(a,2,1)$, when $a \equiv 0,4 \pmod 6$ (corresponding to the Steiner triple systems), and using Singer difference sets for the case $(a, b, c)$, where $a$ and $c$ are given, and $b$ is sufficiently large. A few other small cases are also given.

Atkinson and van Ditmarsch~\cite{AD09} extend these notions to include a new axiom, CA4, which roughly corresponds to our notion of perfect 1-security. That is, the authors recognize the possibility of card occurrence bias in a good announcement, which gives Cathy an advantage in guessing Alice's hand. Axiom CA4 introduces the requirement that, in the set of hands Cathy knows are possible for Alice, each card Cathy does not hold occurs a constant number of times. In this setting, the authors use binary designs to construct a good announcement (also satisfying CA4) for parameters of the form $(2^{k-1},2^{k-1}-1,1)$, where $k \geq 3$. Atkinson and van Ditmarsch also consider the problem of unbiasing an announcement by applying a protocol that takes the existence of bias into account. An example of two possible methods for achieving this are given for the parameter set $(3,3,1)$. We remark that our approach is much simpler and yields nice solutions for the $(3,3,1)$ case. In particular, we require fewer announcements and thereby less communication complexity.

Albert et al.~\cite{ACDFJS11} investigate both the problem of communicating the entire hand (or state information) and communicating a secret bit. In effect, their notion of \emph{card/state safe} is similar to our notion of weak 1-security. The analysis includes a sum announcement protocol for the case $(k,k,1)$, where $k \geq 3$; that is, both players announce the sum of their cards modulo $2k+1$. In addition, Albert et al.\ show that \emph{state safe} implies \emph{bit safe}, and pose the interesting open question of whether a protocol for sharing a secret bit implies the existence of a protocol for sharing states/card deals.

Cord\`{o}n-Franco et al.~\cite{CDFJS} focus on the case c = 1, and present a protocol in which Alice and Bob announce the sum of their hands modulo a given (public) integer. The authors deal with the case of
the modulus being either $n$ (the size of the deck) or the least prime $p$ larger than $n$, and show that, by choosing one of these protocols as appropriate, deals of the form $(a, b, 1)$ are secure (in the weak 1-secure sense) and informative. That is, Alice and Bob learn each other's cards, but Cathy does not know any of Alice or Bob's cards afterwards.

\section{Conclusion and Open Problems}
\label{sec: Conclusion}

We have presented the first formal mathematical presentation of the generalized Russian cards problem, and have provided rigorous security definitions that capture both basic and extended versions of weak and perfect security notions. Using a combinatorial approach, we are able to give a nice characterization of informative strategies having optimal communication complexity, namely the set of announcements must be equivalent to a large set of $t-(n, a, 1)$-designs, where $t=a-c$. We also provide some interesting necessary conditions for certain types of deals to be simultaneously informative and secure. That is, for deals of the form $(a, b, a-d)$, where $b \geq d-1$ and the strategy is assumed to be perfectly $(d-1)$-secure, we show that $a = d+1$ and hence $c=1$. Moreover, for informative and perfectly $(d-1)$-secure deals of the form $(d+1, b, 1)$ satisfying $b \geq d-1$, every announcement must necessarily be a $d-(n, d+1, 1)$-design.

There are many open problems in the area, especially for deals with $c > 1$. An interesting question is whether we can achieve generalizations of Theorems~\ref{Steiner}~and~\ref{main.thm} without assuming $(d-1)$ security. That is, we wish to study the case of deals satisfying $c > 1$, where perfect 1-security holds. In particular, it is unclear if there even exist protocols that are simultaneously informative for Bob and perfectly 1-secure against Cathy for deals with $c > 1$.

\end{document}